\documentclass[a4paper,12pt]{amsart}
\usepackage{amssymb}
\usepackage[latin1]{inputenc}
\usepackage{amsmath}
\usepackage{xcolor}
\usepackage{graphics}
\usepackage{soul}
\usepackage[dvips]{graphicx}
\newtheorem{thm}{Theorem}

\newtheorem{cor}{Corollary}

\def\C{{\mathbb C}}

\def\D{{\mathbb D}}

\makeatletter

\begin{document}

\title {A Characterization of concave mappings}

\author{V. Bravo\and R. Hern\'andez   \and O. Venegas }
\thanks{}
\address{Facultad de Ingenier\'ia y Ciencias\\
Universidad Adolfo Ib\'a\~nez\\
Av. Padre Hurtado 750, Vi\~na del Mar, Chile.}
\email{victor.bravo.g@uai.cl}

\address{Facultad de Ingenier\'ia y Ciencias\\
Universidad Adolfo Ib\'a\~nez\\
Av. Padre Hurtado 750, Vi\~na del Mar, Chile.}
\email{rodrigo.hernandez@uai.cl}

\address{Departamento de Ciencias Matem\'{a}ticas y F\'{\i}sicas. Facultad de Ingenier\'{\i}a\\ Universidad Cat\'olica de
Temuco.} \email{ovenegas@uct.cl}

\thanks{The
authors were partially supported by Fondecyt Grants \# 1190756.
\endgraf  {\sl Key words:} {Concave mappings, Schwarzian derivative, Caratheodory class.}
\endgraf {\sl 2020 AMS Subject Classification}. Primary: 30C45, 30C55; \,
%32H02, 32A17; \,
Secondary: 31A10.}
%\address{}\email{}
%\address{}\email{}

\begin{abstract} This study focuses on Concave mappings, a class of univalent functions that exhibit a unique property: they map the unit disk onto a domain whose complement is convex. The main objective of this work is to characterize these mappings in terms of the real part of the expression $1 +zf''(z)/f'(z)$, considering scenarios where the omitted convex domain is either bounded or unbounded. In the case of a bounded convex domain, we investigate the pivotal role played by the Schwarzian derivative and the order of the functions in understanding the behavior and properties of these mappings.    
\end{abstract}

\maketitle

\section{Introduction}

A complex function defined in the unit disk is said to be concave as opposed to a convex function, that is, those functions that map the unit disk into the complement of a convex domain. In other words, the set of points it omits forms a convex domain in the complex plane. From this definition, we can conclude that these functions, or more precisely, the image domain $f(\mathbb{D})$ cannot be a bounded set and therefore, either inside or on the boundary, $f$ has a pole. Since we consider these functions to be univalent, this pole must be simple. This type of functions are the subject of study in this manuscript.

The two scenarios mentioned above, one with the pole inside the disk, and the other, when the pole is on the boundary, are treated differently and characterized ad-hoc. In the first case, we consider conformal functions that omit an unbounded convex domain and have the pole at $z=1$. This set of functions is denoted as $Co$. While it is true that characterizations of this set of functions exist in terms of the real part of $1+zf''(z)/f'(z)$, a particular case of these types of functions are those that have an angle at infinity less than or equal to $\pi\alpha$, and they are referred to in the literature as $Co(\alpha)$. Here, we propose a new characterization of functions in $Co$ using the well-known differential operators, Schwarzian derivative and the order operator, much like what happens with the extensively studied convex functions defined in the unit disk. The geometric significance of the Schwarzian derivative as a tool in the context of the geometric theory of functions is well known, and in this case, it plays a role equivalent to what it does for convex functions, and this is where we see the specific value of these notes. Section 2 is dedicated to this task, particularly focusing on the extreme case of equality in the obtained characterization. Subsection 2.2 deals with the aforementioned functions in the class $Co(\alpha)$ and their characterization in terms of the real part, which extends the known results up to now.

On the other hand, when the (simple) pole is in the interior, the omitted convex domain is bounded. Several authors have addressed this type of functions by considering the pole to be at $z=p$, where $p\in[0,1)$, and referring to this set as $Co(p)$. A particular case of study is when considering $p=0$, and it has indeed been characterized equivalently to the case of convex functions, except that in this case, the real part of $1+zf''(z)/f'(z)$ is negative. Note that this operator is well-defined when $f$ has a simple pole at the origin. In Section 3, we study this class of concave functions, provide a characterization that includes the aforementioned case, and calculate the value of the Schwarzian derivative for any function within this family. The section concludes with the case of equality in this characterization.

In this manuscript, we study some characterizations of this two classes of concave mappings, $Co(p)$ with $p\in[0,1)$ and $Co$, and the Schwarzian derivative for this mappings. 

\section{On the analytic concave mappings: The class $Co$.}

Let $\mathbb{D}$ denote the open unit disk in the complex plane, and let $\mathbb{C}$ be the complex plane. Consider analytic mappings $f:\mathbb{D}\to\mathbb{C}$ into the complex plane, such that the image of $f$ lies in the complement of a convex domain. Furthermore, let $f(1)=\infty$ be a normalization condition for these mappings. The set of all such functions is denoted by $Co$. The class called $Co(\alpha)$ consists of analytic concave mappings that are normalized such that $f(0) = 0$, $f'(0) = 1$, and $f(1) = \infty$. The opening angle of $f(\mathbb{D})$ at infinity is then bounded by $\pi\alpha$, where $\alpha \in (1,2]$. For further details on these classes, refer to \cite{AW05} as well as \cite{BH12, BPW12}.

In \cite{AW02} the authors shows that if $f\in Co$ then the function $\phi$ defined by the formula \begin{equation}\label{phi} \phi(z)=z+\dfrac{2f'(z)}{f''(z)}\end{equation} is a holomorphic mapping defined in $\D$ and such that $|\phi(z)|\leq 1$, $z=1$ is the unique fixed point, and $\phi'(1)\in[0,1/3]$.

The Schwarzian derivative for locally univalent meromorphic functions $f$, is a well-known differential operator given by $$Sf=\left(\dfrac{f''}{f'}\right)'-\frac12\left(\dfrac{f''}{f'}\right)^2.$$  Is a mathematical tool to measure the distortion caused by a complex mapping. It provides a quantitative measure of how a mapping deforms angles and shapes locally and captures the second-order behavior of $f(z)$. In fact, $f$ maps conformally the unit disk onto a convex domain if and only if $$|Sf(z)|(1-|z|^2)^2+2\left|\dfrac{\varphi(z)-\overline{z}}{1-z\varphi(z)}\right|^2\leq 2,$$ when $$\varphi(z)=\dfrac{f''(z)/f'(z)}{2+zf''(z)/f'(z)}.$$ The reader can be found this characterization in \cite{CDO11}.

Let $f:\D\to \C$ be a conformal mapping. In \cite{CP07}, the authors introduce the definition of inferior order of $f$ as $$\inf_{z\in\D}|A_f(z)|=\inf_{z\in\D}\frac12\left|(1-|z|^2)\frac{f''}{f'}(z)-2\overline z\right|.$$ In the same paper, they proved that $f$, with the normalization that $f(1)=\infty$, is concave if and only if $\inf |A_f|=1$.The reader can see this result as equivalent to the well-known fact that a function is convex if and only if $\sup |A_f|=1$. Likewise, this is also equivalent to $$(1-|z|^2)^2|Sf(z)|+2|A_f(z)|\leq 2,$$ see \cite[eq. (9)]{CDO11}. The following results, is in this sense, can be understood as the generalization of this inequality.

\begin{thm} Let $f:\D\to\C$ be a conformal mapping with $f(1)=\infty$, then 
$f\in Co$ if and only if \begin{equation}\label{thm 1}|Sf(z)|(1-|z|^2)^2+2\leq 2|A_f(z)|^2.\end{equation}
\end{thm}

\begin{proof} Since $f\in Co$ we have that there exists a holomorphic mapping $\phi:\D\to\overline{\D}$ such that the equation (\ref{phi}) holds (see \cite{AW02}). Thus, $f''/f'(z)=2/(\phi(z)-z)$, which implies that $Sf(z)=-2\phi'(z)/(\phi(z)-z)^2.$ From where, and using the Schwarz's Lemma, we have that $$|Sf(z)|(1-|z|^2)^2\leq \frac{2(1-|\phi(z)|^2)(1-|z|^2)}{|\phi(z)-z|^2}=2\left(\frac{|1-\overline{z}\phi(z)|^2}{|\phi(z)-z|^2}-1\right).$$ By equation (\ref{phi}), a direct computation shows that $$\frac{1-\overline z\phi(z)}{\phi(z)-z}=\frac12\left((1-|z|^2)\frac{f''}{f'}(z)-2\overline z\right)=A_f(z),$$ which completes the proof of the necessity part. Now, suppose that $|Sf(z)|(1-|z|^2)^2+2\leq 2|A_f(z)|^2$, then $|A_f(z)|\geq 1$ which is equivalent with the function $f$ is concave (see \cite[Thm. 4 part (ii)]{CP07}). Then the proof is completed.
\end{proof}

\subsection{The case of equality in Theorem 1} The equality in (\ref{thm 1}) for some $z\in\D$, implies that $\phi$ satisfies the equality case in the Schwarz's Lemma, therefore $\phi$ is an automorphism of the unit disk given by
$\phi(z)=\nu (a-z)/(1-\overline{a}z)$ for some $a\in\D$ and $|\nu|=1$. Since $\phi$ has a fixed point at $z=1$ we have that $\nu=-(1-\overline{a})/(1-a)$, and $\phi'(1)=(1-|a|^2)/|1-a|^2\in[0,1/3]$. Moreover, $$\dfrac{f''}{f'}(z)=\dfrac{2(1-\overline{a}z)}{\nu a-(\nu+1)z+\overline{a}z^2}.$$ Observe that if $a=0$ then $\phi(z)=z$ which says that all points are fixed point, which is a contradiction with Theorem 2 part (iii) in \cite{AW02}. Since $a\neq 0$,  $\nu a/\overline{a}\neq 1$, in other case, $\mbox{Re}\{a\}=|a|^2$ which implies that $\phi'(1)=1$. Therefore, we have that $$\dfrac{f''}{f'}(z)=\dfrac{2(1-\overline{a}z)}{\overline{a}(z-1)(z-\lambda)}=\dfrac{2(1-\overline{a})}{\overline{a}(1-\lambda)}\cdot\dfrac{1}{z-1}+\dfrac{2(\nu a-1)}{\overline{a}(1-\lambda)}\cdot\dfrac{1}{z-\lambda}, \quad \lambda=\nu a/\overline{a}.$$Then $$f'(z)=A(z-1)^{-b-2}(z-\lambda)^b,$$ for some complex constant $A$ and $b=2(\nu a-1)/(\overline{a}-\nu a)=(1-|a|^2)/(|a|^2-\mbox{Re}\{a\})$. A straightforward computation and the fact that $\phi'(1)\in[0,1/3]$, shows that $b=2\phi'(1)/(1-\phi'(1))\in[0,1]$. In this case, $$f(z)=A\left(\dfrac{z-\lambda}{z-1}\right)^{1+b}+B,$$ for some constant $B$. Observed that $f$ is an affine maps of the unit disk into a concave domain. This domain is defined as the complement of an angle less than or equal to $\pi$ centered at the origin.   

\subsection{The class of $Co(\alpha)$.} In their seminal paper \cite{AW02}, the authors introduced a noteworthy class of mappings, characterized by a specific normalization condition. This condition ensures that the angle at infinity in $f(\mathbb{D})$ is at most $\alpha\pi$, where $\alpha$ lies in the interval $[1,2]$. Notably, when $\alpha=1$ the mapping reduces to the well-known half-plane mapping $\ell(z)=z/(1-z)$ or affine transformation of this mapping. These functions are commonly referred to as $Co(\alpha)$ mappings. For a more comprehensive understanding of this family, interested readers can refer to the following references: \cite{AW05, BH12, BPW09, BPW10, BPW12}. 

In particular, the author in \cite{BPW12} demonstrates that $(1-|z|^2)^2|Sf(z)|\leq 2(\alpha^2-1)$, which can be derived using our last theorem and the fact that $\sup |A_f|=\alpha$ (see \cite{CP07}) for any mapping belonging to this class. In \cite{AW05} the authors proved (see Theorem 2) that, for $\alpha\in(1,2]$,  $f\in Co(\alpha)$ if and only if \begin{equation}\label{Co(alpha)} \text{Re}\left\{\dfrac{\alpha+1}{2}\cdot\dfrac{1+z}{1-z}-1-z\dfrac{f''}{f'}(z)\right\} >0.\end{equation} Similarly to the case of convex functions, this last inequality can be improved in the sense that not only is the real part of this expression positive, but it is also greater than or equal to a positive expression involving the same elements. This is the goal of the next result.

\begin{thm}
$f\in Co(\alpha)$ with $\alpha\in(1,2]$ if and only if $$\textup{Re}\left\{\dfrac{\alpha+1}{2}\cdot\dfrac{1+z}{1-z}-1-z\dfrac{f''}{f'}(z)\right\}\geq \dfrac{1}{2(\alpha-1)}\left|\dfrac{f''}{f'}-\dfrac{\alpha+1}{1-z}\right|^2(1-|z|^2).$$
\end{thm}

\begin{proof} By inequality (\ref{Co(alpha)}), there exists a Schwarz's map $\omega$ such that $$\dfrac{\alpha+1}{2}\cdot\dfrac{1+z}{1-z}-1-z\dfrac{f''}{f'}(z)=\frac{\alpha-1}{2}\cdot\frac{1+\omega}{1-\omega}.$$ Since $\omega(z)=z\varphi(z)$ for some holomorphic mapping $\varphi:\D\to\D$, we have that $$E:=\frac{1}{\alpha-1}\left(\frac{\alpha+1}{1-z}-\frac{f''}{f'}(z)\right)=\frac{\varphi(z)}{1-z\varphi(z)},$$ which is equivalent to $$\varphi(z)=\frac{E}{1+zE}.$$ Since $|\varphi(z)|\leq 1$ it follows that $$1+2\text{Re}\{zE\}\geq |E|^2(1-|z^2|).$$ But, $$\dfrac{\alpha+1}{2}\cdot\dfrac{1+z}{1-z}-1-z\dfrac{f''}{f'}(z)=\frac{\alpha-1}{2}\left(1+2zE\right),$$ using the last inequality, we conclude as desire. 
\end{proof}

\section{On the meromorphic concave mappings: The class $Co(p)$.}
When this pole is located within $\mathbb{D}$, a well-known class of such mappings arises, known as $Co(p)$, which is normalized with respect to the pole at $p \in [0,1)$. These classes have been extensively studied in \cite{AW02, AW05, APW04, APW06}. Moreover, in \cite{CDO12}, the authors provide a characterization of concave functions with a pole at the origin, normalized as follows: $f(z)=1/z+b_0+b_1z+\cdots$. Now, in general, let $f$ be a concave mapping with a simple pole at $p\in[0,1)$ normalized to \begin{equation}\label{Co(p)} f(z)=\dfrac{1}{z-p}+b_0+b_1(z-p)+\cdots, \quad z\in\D.\end{equation} This class of mappings are called $Co(p)$.  This functions satisfies that\begin{equation*} \mbox{Re}\left\{1+z\dfrac{f''}{f'}(z)+\dfrac{2p}{z-p}-\dfrac{2pz}{1-pz}\right\}<0.\end{equation*} The reader can found more information in \cite{APW04,BPW09,CDO12,Wi04}. In particular, the authors in \cite{CDO12} showed that when the pole at the origin, this functions can be characterized as $$\mbox{Re}\left\{1+z\dfrac{f''}{f'}(z)\right\}<0.$$ Furthermore, within the context of the same paper, it is demonstrated that this inequality can be further refined through, $f\in Co(0)$ if and only if \begin{equation}\label{improve-Co(0)}
\mbox{Re}\left\{1+z\dfrac{f''}{f'}(z)\right\}\leq -\dfrac{1}{4}(1-|z|^4)\left|z\frac{f''}{f'}(z)\right|^2.\end{equation} The equality at a specific point is attained if and only if the dilation function $\omega(z)=\lambda z^2$, where $\lambda$ is an arbitrary unimodular constant. Considering the normalization that the residue at the pole located at $z=0$ admits only one possibility, namely $\lambda=1$. Consequently, in this particular scenario, equality at a point distinct from the origin arises solely when the function $f$ is defined as follows: $$f(z)=\frac{1}{z}+a_0+z.$$  

The subsequent theorem is presented in \cite{CDO12}; however, it is not explicitly delineated. The proof outlined herein significantly diverges from the rationale presented in that publication.

\begin{thm}
$f\in Co(0)$ if and only if $|Sf(z)|(1-|z|^2)^2+2(2|\varphi|+1)|\Phi|^2\leq 2(2|\varphi|+1)$, where $$\varphi(z)=\frac{z+2f'(z)}{z^3f''(z)},\quad \quad  \mbox{and} \quad \quad \Phi(z)=\frac{\overline{z}-z\varphi(z)}{1-z^2\varphi(z)}.$$
\end{thm}

\begin{proof} Let $f$ be a normalized function in $Co(0)$ given by $f(z)=\dfrac{1}{z}+b_0+b_1z+b_2z^2+b_3z^3+\cdots$. Then 
$z\dfrac{f''}{f'}(z)=-2-2b_1z^2-2b_1z^3-6b_2z^3-\cdots$ hence $z\dfrac{f''}{f'}(z)=\dfrac{2}{\omega-1}$, therefore $$z\dfrac{f''}{f'}(z)+2=\dfrac{2\omega}{\omega-1}.$$

From $z\dfrac{f''}{f'}(z)=\dfrac{2}{\omega-1}$ we have that $\dfrac{f''}{f'}(z)=\dfrac{2}{z(\omega-1)}$, and also $$Sf(z)=\left(\dfrac{2}{z(\omega-1)}\right)'-\dfrac{1}{2}\left(\dfrac{2}{z(\omega-1)}\right)^2=\dfrac{-2(\omega+z\omega')}{z^2(\omega-1)^2}.$$

There exists a holomorphic mapping $\varphi:\D\to\D$ such that $\omega(z)=z^2\varphi(z)$ from where we obtain that $\omega'=2z\varphi+z^2\varphi'$ and therefore $$Sf(z)=\dfrac{-6\varphi-2z\varphi'}{(1-z^2\varphi)^2}.$$ Consider $\phi$ such that $\phi(z)=z\varphi(z)$ then $\phi'(z)=\varphi(z)+z\varphi'(z)$. Thus 
\begin{eqnarray*}
|Sf(z)|(1-|z|^2)^2 & = & \dfrac{2(1-|z|^2)^2}{|1-z\phi|^2}|2\varphi+\phi'|\\
 & \leq & 2\left(\dfrac{|2\varphi|(1-|z|^2)^2}{|1-z\phi|^2}+\dfrac{|\phi'|(1-|z|^2)^2}{|1-z\phi|^2}\right)\\
& \leq & 2\left(\dfrac{|2\varphi|(1-|z|^2)^2}{|1-z\phi|^2}+\dfrac{(1-|\phi|^2)(1-|z|^2)}{|1-z\phi|^2}\right).
\end{eqnarray*} Define $\Phi:\D\to \C$ by $\Phi(z)=\dfrac{\overline{z}-\phi(z)}{1-z\phi(z)}$ hence $|\Phi(z)|^2<1$ and $1-|\Phi(z)|^2=\dfrac{(1-|\phi|^2)(1-|z|^2)}{|1-z\phi|^2}$. Then

\begin{eqnarray*}
|Sf(z)|(1-|z|^2)^2 & \leq & 2\left(\dfrac{|2\varphi|(1-|z|^2)^2}{|1-z\phi|^2}+ 1-|\Phi|^2\right)\\
 & = & 2\left(\dfrac{2|\varphi|(1-|z|^2)}{1-|\phi|^2}\right)(1-|\Phi|^2)\\
 & \leq & 2(2|\varphi|+1)(1-|\Phi|^2),
\end{eqnarray*} which is equivalent to $$|Sf(z)|(1-|z|^2)^2+2(2|\varphi|+1)|\Phi|^2\leq 2(2|\varphi|+1).$$
\end{proof}

This result, can be understood as an extension to equation (9) in \cite{CDO11}.

\begin{cor}
$f\in Co(0)$ then $|Sf(z)|(1-|z|^2)^2\leq 6$.
\end{cor}

\begin{proof} Since $\varphi:\D\to\D$ we have that the right side is less than or equal to 6.
\end{proof}

\begin{thm} Let $f:\D\to\C$ satisfying equation (\ref{Co(p)}). Then $f\in Co(p)$, $p\in [0,1)$ if and only if $$\textup{Re}\left\{1+z\dfrac{f''}{f'}(z)+q\right\}\leq -\dfrac{(1-|z|^2)(1+2a_p|z|+|z|^2)}{4(1+a_p|z|)^2}\left|z\dfrac{f''}{f'}(z)+q\right|^2 ,$$ where $a_p=\dfrac{2b_1p^2+O(p^3)}{-1+b_1p^2+O(p^3)}$ and $q=\dfrac{2p}{z-p}-\dfrac{2pz}{1-pz}$.
\end{thm}

\begin{proof} We known that $f\in Co(p)$ if and only if the real part of $$M(z)=1+z\dfrac{f''}{f'}(z)+\dfrac{2p}{z-p}-\dfrac{2pz}{1-pz}$$ is negative. Since $z+2f'(z)/f''(z)$ has a unique fixed point at $p$ (see \cite[Thm. 2]{AW02}), we have that $1+zf''(z)/f'(z)$ is equal to $1$ when $z=0$, thus $M(0)=-1$. Therefore, exists a Schwarz's mapping $\omega$ such that $$1+z\dfrac{f''}{f'}(z)+\dfrac{2p}{z-p}-\dfrac{2pz}{1-pz}=\dfrac{\omega+1}{\omega-1}.$$ Then $\omega=(M+1)/(M-1)$, and $\omega(z)=z\varphi(z)$ where $\varphi:\D\to\D$ given by \begin{equation}\label{varphi}\varphi(z)=\dfrac{\dfrac{f''}{f'}(z)+\dfrac{2}{z-p}-\dfrac{2p}{1-pz}}{z\dfrac{f''}{f'}(z)+\dfrac{2p}{z-p}-\dfrac{2pz}{1-pz}}=\dfrac{(z-p)\dfrac{f''}{f'}(z)+2-\dfrac{2p(z-p)}{1-pz}}{z(z-p)\dfrac{f''}{f'}(z)+2p-\dfrac{2pz(z-p)}{1-pz}}.\end{equation} A straightforward calculations shows that $$\varphi(0)=\dfrac{2b_1p^2+O(p^3)}{-1+b_1p^2+O(p^3)}=a_p,$$ which tends to 0 when $p$ goes to 0. Now, $$|M+1|^2\leq |z|^2|\varphi(z)|^2|M-1|^2,$$ which is equivalent to $$\left|2+z\frac{f''}{f'}(z)+\dfrac{2p}{z-p}-\dfrac{2pz}{1-pz}\right|^2\leq |z|^2|\varphi(z)|^2\left|z\frac{f''}{f'}(z)+\dfrac{2p}{z-p}-\dfrac{2pz}{1-pz}\right|^2,$$ or $$4+4\text{Re}\left\{z\frac{f''}{f'}(z)+\dfrac{2p}{z-p}-\dfrac{2pz}{1-pz}\right\}\leq\left|z\frac{f''}{f'}(z)+\dfrac{2p}{z-p}-\dfrac{2pz}{1-pz}\right|^2(|z\varphi(z)|^2-1),$$ from where $$\text{Re}\left\{1+z\frac{f''}{f'}(z)+\dfrac{2p}{z-p}-\dfrac{2pz}{1-pz}\right\}\leq-\frac{1-|z|^2|\varphi(z)|^2}{4}\left|z\frac{f''}{f'}(z)+\dfrac{2p}{z-p}-\dfrac{2pz}{1-pz}\right|^2.$$ Then, using that the generalization of Schwarz's Lemma, we have that $$\text{Re}\left\{1+z\frac{f''}{f'}(z)+q\right\}\leq-\frac{1-|z|^2\left(\dfrac{|a_p|+|z|}{1+|a_p||z|}\right)^2}{4}\left|z\frac{f''}{f'}(z)+q\right|^2,$$ where $q=\dfrac{2p}{z-p}-\dfrac{2pz}{1-pz}$. Thus, the proof is completed.
\end{proof}

Observed that the inequality (\ref{improve-Co(0)}) is a particular case of this Theorem considering $p=0$, in which case $a_p=0$, or equivalently $\varphi(0)=0$. Since $\varphi(z)$ is given by equation (\ref{varphi}), we have that $\varphi(0)=-a_2-\dfrac{1}{p}-p$, then $|a_2|\leq 1/p+p$, and the equality case attained by $f(z)=e^{i\theta}k_p(e^{-i\theta}z)$, where $$k_p(z)=\dfrac{z}{(1-z/p)(1-pz)},$$ which maps the unit disk into a hole the complex plane minus the real segment $\left[\dfrac{1}{2-1/p-p},\dfrac{-1}{2+1/p+p}\right]$.

%%%%%%%%%%%%%%%%%%%%%%%%%%%%%%%%%%%%%%%%%%%%%%%%%%%%%%%%%%%%%%%%%%%%%%%%%%%%%%%%%%%%%%%%%%%%%%


\begin{thebibliography}{11}
\bibitem{AW02} F.G. Avkhadiev, and K.-J. Wirths, Complex hole produce lower bounds for coefficients, Complex variable, Theory and Application: An International Journal. 47(7) (2002), pp. 553--563.\medskip

\bibitem{AW05} F.G. Avkhadiev, and K.-J. Wirths, Concave schlicht functions with bounded opening angle at infinity, Lobachevskii J. Math. 17 (2005), pp. 3--10.\medskip

\bibitem{APW04} F.G. Avkhadiev, Ch. Pommerenke, and K.-J. Wirths, On the coefficients of concave univalent functions, Math. Nachr. 271 (2004), pp. 3--9.\medskip

\bibitem{APW06} F.G. Avkhadiev, Ch. Pommerenke, and K.-J. Wirths, Sharp inequalities for the coefficients of concave schlicht functions, Comment. Math. Helv. 81 (2006), pp. 801--807.\medskip

\bibitem{BH12} B. Bhowmik, On concave univalent functions, Math. Nachr. 285 (2012), pp. 606--612.\medskip

\bibitem{BPW09} B. Bhowmik, S. Ponnusamy, and K.-J. Wirths, Concave functions, Blaschke products, and polygonal mappings, Siberian Mathematical Journal 50(4) (2009), pp. 609--615.\medskip

\bibitem{BPW10} B. Bhowmik, S. Ponnusamy, and K.-J. Wirths, Characterization and the pre-Schwarzian norm estimate for concave univalent functions, Monatsh. Math. 161 (2010), pp. 59--75.\medskip

\bibitem{BPW12} B. Bhowmik, and K.-J. Wirths, A sharp bound for the Schwarzian derivative of concave functions, Coll. Math. 128(2) (2012), pp. 245--251.\medskip

\bibitem{CDO11} M. Chuaqui, P. Duren, and B. Osgood, Schwarzian derivative of convex mappings, Ann. Acad. Sci. Fenn., 36 (2011), pp. 449--460.\medskip                                                    
\bibitem{CDO12} M. Chuaqui, P. Duren, and B. Osgood, Concave conformal mappings and pre-vertices of Schwarz-Christoffel mappings, Proc. Am. Math. Soc. 140(10) (2012), pp. 3495--3505.\medskip

\bibitem{CP07} L. Cruz, and Ch. Pommerenke, On concave univalent functions, Complex Var. Elliptic Equ. 52(2-3) (2007), pp. 153--159.\medskip

\bibitem{Wi04} K.-J. Wirths, A proof of the Livingston conjecture for the fourth and the fifth coefficient of concave univalent functions, Ann. Polon. Math. 83(1) (2004), pp. 87--93.\medskip





\end{thebibliography}
\end{document}